\documentclass[11pt,reqno]{amsart}
 \usepackage{graphicx}
\usepackage{amscd,amsmath,amsopn,amssymb,amsthm,multicol}
\usepackage[color,matrix, all, 2cell]{xy}
\usepackage{amscd}
\usepackage{lscape}
\usepackage{slashed}
\usepackage{graphicx}
\usepackage{setspace}
\usepackage{upgreek}
\usepackage{textgreek}
\usepackage{enumerate}
\usepackage{color}
\usepackage{lscape}
\usepackage{tikz}
\usepackage{multirow}
\usepackage{cancel}
\usepackage{soul}
\usepackage{harmony}
\usepackage{comment}
\usepackage{wasysym}
\usepackage{mathrsfs}
\usepackage{mathtools}

\numberwithin{equation}{section}

\DeclareMathAlphabet{\mathscrbf}{OMS}{mdugm}{b}{n}

\DeclareMathOperator{\Id}{Id}
\DeclareMathOperator{\tr}{tr}

\DeclareMathOperator{\Ric}{Ric}

\DeclareMathOperator{\vol}{vol}

\DeclareMathOperator{\dd}{d}
\DeclareMathOperator{\Div}{div}

\newcommand{\bb}{\mathbb}

\newcommand{\cal}{\mathcal}

\DeclareMathOperator{\G}{\mathsf{G}}

\DeclareMathOperator{\Scal}{Scal}

\DeclareMathOperator{\Ed}{End}

\theoremstyle{plain}
\newtheorem{lemma}{Lemma} [section]
\newtheorem{theorem}[lemma]{Theorem}
\newtheorem{corol}[lemma] {Corollary}

\theoremstyle{definition}
\newtheorem{definition}[lemma] {Definition}

\newtheorem{remark}[lemma] {Remark}
\newtheorem*{remark*}{Remark}

\definecolor{dark}{rgb}{0.18,0.18,0.68}
\definecolor{mydark}{rgb}{0.78,0.08,0.08}
\definecolor{crew}{rgb}{0.2,0.5,0.2}
\definecolor{mmg}{rgb}{0.31,0.50,0.23}
\definecolor{dblue}{rgb}{0.01,0.01,0.44}
\definecolor{red}{rgb}{0.57,0.11,0.15}
\definecolor{cobalt}{RGB}{10,19,111}
\usepackage[colorlinks,citecolor=cobalt,linkcolor=cobalt,urlcolor=cobalt,pdfpagemode=UseNone,backref = page]{hyperref}

 \input ulem.sty
 
\begin{document}  
 

\title{A note on the volume of  $\nabla$-Einstein manifolds with skew-torsion}
\author{Ioannis Chrysikos}
 \address{University of Hradec Kr\'alov\'e, Faculty of Science, Rokitansk\'eho 62, 500 03 Hradec Kr\'alov\'e, Czech Republic}
 \email{ioannis.chrysikos@uhk.cz}
 

 \begin{abstract}
We study the volume of compact Riemannian manifolds which are Einstein with respect to a metric connection with (parallel) skew-torsion. We provide  a result for the sign of the first variation of the volume in terms of the corresponding scalar curvature. This generalizes a result of M.~Ville \cite{Vil} related with the first variation of the volume on a compact Einstein manifold.  
 \end{abstract}
\maketitle  

\section*{Introduction}

Consider a compact Einstein manifold $(M^{n}, g)$   and let $\cal{M}$ be the space of  Riemannian metrics on $M$.  A smooth variation of the metric $g$ in the direction of some arbitrary  but fixed smooth symmetric $(2, 0)$-tensor $h\in\Gamma(S^{2}T^{*}M)=:\cal{S}^{2}(M)$ on $M$, is a  smooth curve  $g(t) : (-\epsilon, \epsilon)\to \cal{M}$ 
with $g(0)=g$ and $\dot{g}(0)=\frac{d}{dt}\big|_{t=0}g(t)=h$. We view the scalar curvature  related with the Levi-Civita connection $\nabla^{g}$ as a functional $\cal{M}\to C^{\infty}(M)$, defined by $g\mapsto \Scal_{g}$, 
and in a line with the notation of \cite{Bes}, we denote its differential $\Scal_{g}'$ at $g$  in the direction of $h$ (or the first variation in the direction of $h$) by
\[
\Scal_{g}'(h):=\frac{d}{dt}\Big|_{t=0}\Scal_{g+th}=\frac{d}{dt}\Big|_{t=0}\Scal_{g(t)}\,.
\]
  By Ville \cite{Vil} it is known that the first variation of  the volume $\vol_{g}(M)$ of $(M^n, g)$ enjoys some nice properties which can be expressed in terms of the  scalar curvature. In particular,   
  \[
  \int_{M}(\Scal_{g}'h )dV_{g}=-\frac{2}{n}\Scal_{g}\vol_{g}(M)'h \,,
  \]
where $dV_{g}$ is the volume element,  which  can be also written as
  \[
  \int_{M}\frac{d}{dt}\Big|_{t=0}\Scal_{g(t)}dV_{g}=-\frac{2}{n}\Scal_{g}\frac{d}{dt}\Big|_{t=0}\vol_{g(t)}(M)\,.
  \] 
Here,   $\Scal_{g(t)}$ and $\vol_{g(t)}(M)$ denote the scalar curvature and volume of $g(t)$, respectively. Therefore, if $\Scal_{g}=0$, or if  $\vol_{g}(M)'h=0$, then the differential $\Scal_{g}'$ cannot have constant sign, unless it is identically zero (see also \cite[Prop.~1.188]{Bes}). 

\smallskip
Our purpose in this short note  is to extend Ville's results on  Riemannian manifolds which are ``Einstein'' with respect to a metric connection $\nabla$ with skew-torsion $T\in\Omega^{3}(M)$, i.e.
\[
\Ric^{\nabla}_{S}=\frac{\Scal^{\nabla}}{n}g\,, 
\]
where $\Ric^{\nabla}_{S}$ denotes the symmetric part of the Ricci tensor $\Ric^{\nabla}$ associated to $\nabla$ (see below).  Such geometric structures are   called {\it $\nabla$-Einstein manifolds with skew-torsion} and  play a key role in the theory of  non-integrable geometries, due to the so-called {\it characteristic connection} (cf. \cite{FrIv}). This is a metric connection $\nabla$ with skew-torsion as above, i.e. 
\[
\nabla=\nabla^{g}+\frac{1}{2}T
\]
which preserves  the underlying non-integrable geometry and hence it is a natural replacement of the Levi-Civita connection. 
A first systematic study of $\nabla$-Einstein manifolds with skew-torsion, in a variety of different dimensions, was given in \cite{AFer}.   It is remarkable that this work  also provides   the existence of  $\nabla$-Einstein manifolds with skew-torsion, which are {\it not} Einstein with respect to the corresponding metric, e.g.  the Allof-Wallach spaces.  Further such examples were constructed   on Berger spheres in \cite{Draper2}, while similar existence results are known even for non-Einstein Lorentzian metrics.   Notice also that  there are  $\nabla$-Einstein structures with skew-torsion which provide  examples of manifolds  satisfying  important   spinorial equations,  a fact which yields a strong interplay with   string theory  (see e.g. \cite{FrIv,  ABK, Chrys1, Chrys2}).    Hence,  the last decades the geometry  of  $\nabla$-Einstein manifolds with skew-torsion  has  attracted more attention.  For instance, the recent works \cite{Chrys, Draper, Draper2, Chrys3}  present  classification results of  such structures for particular families of homogeneous spaces  (in terms of invariant connections and representation theory). 



\smallskip
At this point  we need to emphasize an important difference in comparison with the classical notion of Einstein manifolds, namely: {\it   $\nabla$-Einstein manifolds with skew-torsion $(M^{n}, g, T)$ may have non-constant scalar curvature $\Scal^{\nabla}$} (\cite{AFer}).  This  comes true since for example in this case the corresponding   {\it Einstein tensor} 
\[
G^{\nabla}:=-\Ric_{S}^{\nabla}+\frac{1}{2}\Scal^{\nabla} g
\]
  is not necessarily divergence free.
In this short note by assuming that the scalar curvature $\Scal^{\nabla}$ is constant, we provide an  extension of  Ville's result for compact  $\nabla$-Einstein manifolds with skew-torsion.  Examples of  $\nabla$-Einstein manifolds  which verify our unique assumption   are those whose torsion form $T$ is $\nabla$-parallel.  Consequently,  one can present  a wealth of examples for which our result makes sense, e.g. 6-dimensional nearly K\"ahler manifols, 7-dimensional weak $\G_2$-manifolds, 7-dimensional 3-Sasakian manifolds and naturally reductive spaces are few of them (see for example \cite{AFer, ABK, Chrys1, Chrys2} and the references therein). 


\section{$\nabla$-Einstein manifolds with skew-torsion}
 Let  $(M^{n}, g)$ be a connected oriented Riemannian manifold. We fix once and for all  a  $g$-orthonormal basis $\{e_{1}, \ldots, e_{n}\}$  of $T_{x}M$ at some point $x\in M$.  
 Recall that a linear connection 
 \[
 \nabla : \Gamma(TM)\mapsto\Gamma(T^{*}M\otimes TM)
 \]
   is said to be metric if $\nabla g=0$. 
The torsion $T : \Gamma(TM)\times\Gamma(TM)\to\Gamma(TM)$ of $\nabla$ is the vector-valued 2-form 
  defined by $T(X, Y):=\nabla_{X}Y-\nabla_{Y}X-[X, Y]$, for any  vector  field $X, Y\in\Gamma(TM)$. $\nabla$ is said to be with  totally anti-symmetric torsion if the induced tensor $T(X, Y, Z):=g(T(X, Y), Z)$ is a 3-form on $M$ and then the following identity holds 
\[
g(\nabla_{X}Y, Z)=g(\nabla^{g}_{X}Y, Z)+\frac{1}{2}T(X, Y, Z) \,.
\]
Thus, in such a case  the  condition $T\in\Omega^{3}(M)$   induces  the dimensional restriction  $n=\dim_{\bb{R}}M\geq 3$. 
 
Let us  denote by 
 \[
\dd_{\nabla} : \Omega^{p}(M)\to\Omega^{p+1}(M)\,,\quad \dd_{\nabla}^* : \Omega^{p}(M)\to\Omega^{p-1}(M)\,,
\]
the differential and co-differential induced by $\nabla$, which are the differential operators defined by $\dd_{\nabla}\omega:=\sum_{i}e_{i}\wedge \nabla_{e_{i}}\omega$ and  $\dd_{\nabla}^{*}\omega:=-\sum_{i}e_{i}\lrcorner  \nabla_{e_{i}}\omega$, respectively. 
 In a line with the  Riemannian case, $\dd_{\nabla}$ and $\dd_{\nabla}^{*}$ are formally adjoint each other, but in general one has  $(\dd_{\nabla})^{2}\neq 0$. 
 The difference between $\dd_{\nabla}T$ and $\dd T$, where we set $\dd\equiv \dd_{g}$, is a 4-form which (up to a factor) will be denoted by $\sigma_{T}$, namely $\dd_{\nabla}T-\dd T=-2\sigma_{T}$. This 4-form $\sigma_{T}$ does not play some explicit role  in this note and we refer   to \cite{FrIv} for alternative definitions. Notice however that  the co-differential of the torsion form $T$ satisfies  $\dd_{\nabla}^{*}T=\dd^{*} T$ and under the assumption $\nabla T=0$ the following hold (see  \cite{AF, FrIv})
    \[
   \dd_{\nabla}^{*}T=0=\dd^{*} T \, , \quad \dd T=2\sigma_{T} \, .
    \]
  For the curvature tensor $R^{\nabla}$ associated to $\nabla$, we adopt the convention  
  \[
  R^{\nabla}(X, Y)Z:=\nabla_{X}\nabla_{Y}Z-\nabla_{Y}\nabla_{X}Z-\nabla_{[X, Y]}Z\,,
  \]
and set $R^{\nabla}(X, Y, Z, W):=g(R^{\nabla}(X, Y)Z, W)$.   In terms of the $g$-orthonormal frame $\{e_1, \ldots, e_n\}$, the Ricci tensor associated to $\nabla$ is given by 
\[
 \Ric^{\nabla}(X, Y)=\sum_{i}g(R^{\nabla}(X, e_i)e_i, Y)=\sum_{i}R^{\nabla}(X, e_i, e_i, Y)\,.
 \]
Moreover, the two Ricci tensors are related by (see e.g. \cite{FrIv})
\[
 \Ric^{\nabla}(X, Y)=\Ric^{g}(X, Y)-\frac{1}{4}S(X, Y)-\frac{1}{2}(\dd^{*}T)(X, Y)\,,
\]
  where $S$ is the symmetric tensor  defined by
  \begin{equation}\label{S}
S(X, Y):=\sum_{i=1}^{n}g(T(e_{i}, X), T(e_{i}, Y))=\sum_{i, j=1}^{n}T(e_{i}, X, e_{j})\cdot T(e_{i}, Y, e_{j}) \,.
\end{equation}
Thus, in contrast to the Riemannian Ricci tensor $\Ric^{g}$, the Ricci tensor of $\nabla$ is not in general symmetric; it decomposes into a symmetric and anti-symmetric part $\Ric^{\nabla}=\Ric_{S}^{\nabla}+\Ric_{A}^{\nabla}$, given by 
\[
\Ric_{S}^{\nabla}(X, Y):=\Ric^{g}(X, Y)-\frac{1}{4}S(X, Y)\, ,\quad \Ric_{A}^{\nabla}(X, Y):=-\frac{1}{2}(\dd^{*}T)(X, Y) \, ,
\]
respectively. As we said above, when $T$ is $\nabla$-parallel, then $\dd^{*}T=0$ and hence $\Ric^{\nabla}=\Ric^{\nabla}_{S}$. Finally, the scalar curvature $\Scal^{\nabla}=\tr\Ric^{\nabla}$ of $(M, g, T)$ with respect to $\nabla$ satisfies the identity     $\Scal^{\nabla}=\Scal^{g}-\frac{3}{2}\|T\|^{2}_{g}$\,.

\begin{definition} 
A triple $(M^{n}, g, T)$ $(n\geq 3)$ is called a {\it $\nabla$-Einstein manifold with   skew-torsion} $0\neq T\in\Omega^{3}(TM)$, or in short, a {\it $\nabla$-Einstein manifold}, if the symmetric part $\Ric_{S}^{\nabla}$ of the Ricci tensor associated to the metric connection $\nabla=\nabla^{g}+\frac{1}{2}T$ satisfies the equation 
\begin{equation}\label{skein}
\Ric_{S}^{\nabla}=\frac{\Scal^{\nabla}}{n}g \, ,
\end{equation}
where $\Scal^{\nabla}$ is the scalar curvature associated to $\nabla$ and $n=\dim_{\bb{R}}M$. If $\nabla T=0$, then $(M, g, T)$ is called a {\it $\nabla$-Einstein manifold with parallel skew-torsion}. 
\end{definition}
In a line  with the case of  Einstein manifolds, in  \cite{AFer} it was shown that  $\nabla$-Einstein manifolds $(M, g, T)$   attain   a variational approach which reads in terms of a {\it generalized scalar curvature functional} $\cal{L}$, given by
\[
  (g, T)\mapsto \int_{M}\Big(\Scal^{\nabla}-2\Lambda\Big)\dd V_{g}=\int_{M}\Big(\Scal^{g}-\frac{3}{2}\|T\|^{2}_{g}-2\Lambda\Big)dV_{g}.
\]
In particular, if  $\Lambda=0$, then the  critical points of $\cal{L}$ are  triples $(M, g, T)$ satisfying $\Ric_{S}^{\nabla}=0$ identically, and conversely. If $\Lambda=0$ and $\nabla T=0$, then the critical points are $\Ric^{\nabla}$-flat manifolds. For $\Lambda\neq 0$,  critical points of $\cal{L}$ are    $\nabla$-Einstein manifolds $(M, g, T)$ with $\Scal^{\nabla}\neq 0$, i.e. $\Ric_{S}=\frac{\Scal^{\nabla}}{n}g$ (so-called {\it strictly $\nabla$-Einstein metrics}), and conversely.

\section{An extension of a result of Ville}
Let us now  prove the analogue of the infinitensimal result of  Ville \cite[Prop.~3]{Vil} about the sign of the first variation of the volume on compact Einstein manifolds (see also Proposition 1.188 in Besse's book \cite{Bes}). The volume of a triple $(M^{n}, g, T)$ will be denoted by ${\rm Vol}_{g}(M):=\int_{M}dV_{g}$.  Also, given    $h_1, h_2\in \cal{S}^{2}(M)$, we shall denote  their inner product with respect to $g$ by
   \[
(h_1, h_2)_g:=\sum_{i, j}h_1(e_i, e_j) h_2(e_i, e_j)\,.
\]
On the other hand, for the square norm of the 3-form $T$ with respect to $g$ we fix the normalization 
 \begin{equation}\label{leng}
\|T\|^{2}_{g}:=\frac{1}{3!}\sum_{i, j}g(T(e_{i}, e_{j}), T(e_{i}, e_{j}))\,.
\end{equation}

\begin{theorem}\label{new1}
Let $(M^{n}, g, T)$ be a compact $\nabla$-Einstein manifold with skew-torsion, i.e. $\Ric_{S}^{\nabla}=\frac{1}{n}\Scal^{\nabla}g$.  Assume that   $\Scal^{\nabla}$  is constant.   Then, the first variation of the volume  ${\rm Vol}_{g(t)}(M)$ with respect to the curve $g(t)\in\cal{M}$ in the direction of $h\in \cal{S}^{2}(M)$, satisfies the relation
\begin{equation}\label{lnew}
-\frac{2\Scal^{\nabla}}{n} \frac{d}{dt}\Big|_{t=0}{\rm Vol}_{g(t)}(M)=\int_{M}\frac{d}{dt}\Big|_{t=0}\Scal^{\nabla}_{g(t)} \ dV_{g},
\end{equation}
where  $\Scal^{\nabla}_{g(t)}:=\Scal^{g(t)}-\frac{3}{2}\|T\|^{2}_{g(t)}$ such that $\Scal^{\nabla}_{g(0)}\equiv\Scal^{\nabla}_{g}=:\Scal^{\nabla}$.
 \end{theorem}
\begin{proof}
Set ${\rm Vol}_{g(t)}(M)=\int_{M}dV_{g(t)}$.  The  first variation of the volume element $dV_{g}$  is given by (see for example \cite{Bar, Ku})
\[
\frac{d}{dt}\Big|_{t=0}dV_{g(t)}=\frac{1}{2}(g, h)_{g}dV_{g}=\frac{1}{2}(\tr_{g}h) \ dV_{g}\,,
\]
 thus   
\begin{equation}\label{volmm}
 \frac{d}{dt}\Big|_{t=0}{\rm Vol}_{g(t)}(M)=\frac{1}{2}\int_{M}\tr_{g}h \ dV_{g}.
\end{equation}
Since $(M, g, T)$ is $\nabla$-Einstein, it follows that
\[
(\Ric_{S}^{\nabla}, h)_{g}=\frac{\Scal^{\nabla}}{n}\tr_{g}h
\]
which is equivalent with the relation
\begin{equation}\label{form1}
(\Ric^{g}, h)_{g}=\frac{\Scal^{\nabla}}{n}\tr_{g}h+\frac{1}{4}(S, h)_{g}.
\end{equation}
Now we  need the first variation of the square norm   of the torsion form in the direction of some $h\in\cal{S}^{2}(M)$. This  is given by (see  \cite{AFer} and see also our  Appendix  for an alternative proof)
\begin{equation}\label{dleng}
\frac{d}{dt}\Big|_{t=0}\|T\|^{2}_{g(t)}=-\frac{1}{6}(S, h)_{g}\,.
\end{equation}
Consequently,  
\begin{equation}\label{intimp}
\int_{M}(S, h)_{g} \ dV_{g}=-6\int_{M}\frac{d}{dt}\Big|_{t=0}\|T\|^{2}_{g(t)} \ dV_{g}\,.
\end{equation}
Finally, let us recall the  first variation of $\Scal^{g}$   (see \cite{Bar, Ku})
 \[
\frac{d}{dt}\Big|_{t=0}\Scal^{g(t)}=\Delta_{g}(\tr_{g}h)+\Div_{g}(\Div_{g}h)-(\Ric^{g}, h)_{g}\,,
\]
where we write $\Div_{g} : \cal{S}^{p}(M)\to \cal{S}^{p-1}(M)$ for   the divergence of a symmetric  tensor field $F$ on $M$,  i.e.  $\Div_{g}(F)(X_1, \ldots, X_{p-1}):=-\sum_{i=1}^{n}(\nabla^{g}_{e_{i}}F)(e_{i}, X_1, \ldots, X_{p-1})$.
Thus, and since $M$ is compact, by  integrating   one computes
\begin{eqnarray*}
\int_{M}\frac{d}{dt}\Big|_{t=0}\Scal^{g(t)} \ dV_{g}&=&-\int_{M} (\Ric_{g}, h)_{g} \ dV_{g}\\
&\overset{(\ref{form1})}{=}&-\int_{M}\frac{\Scal^{\nabla}}{n}\tr_{g}h  \ dV_{g}-\int_{M}\frac{1}{4}(S, h)_{g} \ dV_{g}\\
&\overset{(\ref{volmm}), (\ref{intimp})}{=}&-\frac{2\Scal^{\nabla}}{n} \frac{d}{dt}\Big|_{t=0}{\rm Vol}_{g(t)}(M)+\frac{3}{2}\int_{M}\frac{d}{dt}\Big|_{t=0}\|T\|^{2}_{g(t)} \ dV_{g}\,,
\end{eqnarray*}
or  equivalently 
\[
\int_{M}\frac{d}{dt}\Big|_{t=0}\Big(\Scal^{g(t)}-\frac{3}{2}\|T\|^{2}_{g(t)} \Big) \ dV_{g}=-\frac{2\Scal^{\nabla}}{n} \frac{d}{dt}\Big|_{t=0}{\rm Vol}_{g(t)}(M)\,.
\]
This proves our assertion.
\end{proof}
\begin{remark}
\textnormal{It should be clear that the right hand side of (\ref{lnew}) does not coincide with the variation of the generalized total scalar curvature functional $\cal{L}$ (in the direction of $h$). This only happens when the first variation of the volume element is zero, $\frac{d}{dt}\Big|_{t=0}dV_{g(t)}=0$.}
\end{remark}
As an immediate corollary of Theorem \ref{new1} we get that 
\begin{corol}\label{color1}
Let  $(M^{n}, g, T)$ be  a compact $\nabla$-Einstein manifold with constant $\Scal^{\nabla}$.  If \[
\frac{d}{dt}\Big|_{t=0}{\rm Vol}_{g(t)}(M)=0 \, ,
\] or if $\Scal^{\nabla}=0$, then the differential $(\Scal^{\nabla}_{g})'$   cannot have a constant sign  (unless it is identically zero). \end{corol}

Theorem \ref{new1} shows that \cite[Prop. 1.188]{Bes} holds more general for any compact Riemannian manifold which is $\nabla$-Einstein with respect to a metric connection with {\it parallel skew-torsion}. Indeed, when $\nabla T=0$,  then we can show that the associated  Einstein tensor
\[
G^{\nabla}:=-\Ric_{S}^{\nabla}+\frac{1}{2}\Scal^{\nabla} g=G^{g}+\frac{1}{4}S-\frac{3}{4}\|T\|^{2}_{g}g
\]
is divergence free, and  as in the classical case  for $n\geq 3$ this yields  the constancy of  $\Scal^{\nabla}$ (see also \cite{AFer}). Hence we deduce
\begin{corol}
  Theorem \ref{new1} applies on any compact $\nabla$-Einstein manifold $(M^{n}, g, T)$ $(n\geq 3)$ with $\nabla T=0$.
    \end{corol}

\medskip
 \noindent {\bf Acknowledgements:}  The author acknowledges    Czech Science Foundation  (GA\v{C}R) for support via the  program GA\v{C}R 19-14466Y.

\section{Appendix}
 
Here we present a proof of the  formula (\ref{dleng}), in a  slightly different way than the one discussed  in \cite{AFer}.  

Recall that given any two Riemannian metrics $a, b\in\cal{M}$,  there is   a unique self-adjoint positive definite  endomorphism $B_{a, b}$ on  the tangent space $T_{x}M$ $(x\in M)$, such that 
\[
b(u, v)=a(B_{a, b}u, v),
\]
 for all $u, v\in T_{x}M$.
Hence $d^{a}_{b}:=B_{a, b}^{-1/2}$ is an isometry on $T_{x}M$ and this extends to an isometry on whole tangent bundle $TM$, which we  denote by the same letter,  $d^{a}_{b} : T^{a}M\equiv TM\to T^{b}M\equiv TM$,  $a(X, Y)=b(d^{a}_{b}X, d^{a}_{b}Y)$ for any $X, Y\in\Gamma(TM)$.
 Consider the variation $g(t)\equiv g_{t}=g+th$ of $g$ in the direction of   some symmetric $(2, 0)$-tensor $h$ on $M$. Since there exists  an open neighborhood $I=(-\epsilon, \epsilon)\subset\bb{R}$ of $0\in\bb{R}$  such that $g_{t}\in\cal{M}$ for any $t\in I$, one has  $g(G^{t}_{g, g_{t}}u, w)=g_{t}(u, w)$, 
  for any $u, w\in T_{x}M$, where   $G^{t}_{g, g_{t}}\in\Ed(T_{x}M)$ is the self-adjoint positive definite endomorphism corresponding to $g_{t}$.  Then, by the definition of $g_t$ it follows that 
  \[
  g(G^{t}_{g, g_{t}}u, w)=g_{t}(u, w)=g(u, w)+th(u, w)=g(u, w)+tg(H_{g, h}u, w)=g\big((\Id+tH_{g, h})u, w\big),
  \]
i.e.  $G^{t}_{g, g_{t}}=\Id+tH_{g, h}$ and  hence $d^{g}_{g_t}=(G^{t}_{g, g_t})^{-1/2}=(\Id+tH_{g, h})^{-1/2}$.\footnote{Note that any arbitrary element $h\in\cal{S}^{2}(M)$ induces some endomorphism $H_{g, h}\in\Ed(T_{x}M)$ defined by   $h(u, v)=g(H_{g, h}u, v)$, see for example \cite[p.~134]{FKim}.} Consequently
 \begin{equation}\label{eq1}
\frac{d}{dt}\Big|_{t=0}d^{g}_{g_t}=-\frac{1}{2}(\Id+tH_{g, h})^{-3/2}H_{g, h}\Big|_{t=0}=-\frac{1}{2}H_{g, h}\, .
\end{equation}
Consider now some $g$-orthonormal frame $\{e_{i}\}$  of $M$. Then, the set 
\[\{e_{i}(t):=d^{g}_{g_{t}}(e_{i}) : 1\leq i\leq  n\}
\]
forms a $g(t)$-orthonormal frame such that $e_{i}(0)=d^{g}_{g}(e_i)=\Id(e_i)=e_i$, for any $i$. 
Since $h(e_i, e_j)=g(H_{g, h}e_i, e_j)$ we obtain
\[
H_{g, h}e_i=\sum_{j}g(H_{g, h}e_i, e_j)e_j=\sum_{j}h(e_i, e_j)e_j
\]
and hence (\ref{eq1}) yields the formula
\begin{equation}\label{eq2}
\frac{d}{dt}\Big|_{t=0}d^{g}_{g_t}(e_i)=\frac{d}{dt}\Big|_{t=0}e_{i}(t):=\dot{e}_i(t)\Big|_{t=0}=\dot{e}_{i}(0)=-\frac{1}{2}H_{g, h}e_i=-\frac{1}{2}\sum_{j}h(e_i, e_j)e_j\, .
\end{equation}
Now, by (\ref{leng}) we see that
{\small  \begin{eqnarray*}
 \frac{d}{dt}\|T\|^{2}_{g(t)}&=&\frac{1}{6}\sum_{i, j}\frac{d}{dt}\Big[(g+th)\Big(T\big(e_{i}(t), e_j(t)\big), T\big(e_{i}(t), e_j(t)\big)\Big)\Big]\\
 &=&\frac{1}{6}\sum_{i, j}\dot{g}(t)\Big(T\big(e_{i}(t), e_{j}(t)\big), T\big(e_{i}(t), e_{j}(t)\big)\Big)+\frac{1}{3}\sum_{i, j}g(t)\Big(\frac{d}{dt}T\big(e_{i}(t), e_{j}(t)\big), T\big(e_{i}(t), e_{j}(t)\big)\Big)\\
 &=&\frac{1}{6}\sum_{i, j}\dot{g}(t)\Big(\sum_{k}g\big(T(e_{i}(t), e_{j}(t)), e_{k}(t)\big)e_{k}(t), \sum_{\ell}g\big(T(e_{i}(t), e_{j}(t)), e_{\ell}(t)\big)e_{\ell}(t)\Big)\\
 &&+\frac{1}{3}\sum_{i, j}g(t)\Big(T\big(\dot{e}_{i}(t), e_{j}(t)\big)+T\big(e_{i}(t), \dot{e}_{j}(t)\big), T\big(e_{i}(t), e_{j}(t)\big)\Big)\\
 &=&\frac{1}{6}\sum_{i, j,k, \ell}T\big(e_{i}(t), e_{j}(t), e_{k}(t)\big)\cdot T\big(e_{i}(t), e_{j}(t), e_{\ell}(t)\big)\cdot \dot{g}(t)\big(e_{k}(t), e_{\ell}(t)\big)\\
 &&+\frac{1}{3}\sum_{i, j}g(t)\Big(T\big(\dot{e}_{i}(t), e_{j}(t)\big)+T\big(e_{i}(t), \dot{e}_{j}(t)\big), T\big(e_{i}(t), e_{j}(t)\big)\Big).
 \end{eqnarray*}}
Therefore, due to (\ref{eq2}) and the  tensor $S$ defined in (\ref{S}), the evaluation of the previous relation at $t=0$ gives the result:
{\small\begin{eqnarray*}
 \frac{d}{dt}\Big|_{t=0}\|T\|^{2}_{g(t)}&=&\frac{1}{6}\sum_{k, \ell}S(e_{k}, e_{\ell})h(e_{k}, e_{\ell})-\frac{1}{6}\sum_{i, j}g\big(T(H_{g, h}e_{i}, e_{j}), T(e_{i}, e_{j})\big)-\frac{1}{6}\sum_{i, j}g\big(T(e_{i}, H_{g, h}e_{j}), T(e_i, e_j)\big)\\
 &=&\frac{1}{6}(S, h)_{g}-\frac{1}{6}\sum_{i}S(H_{g, h}e_{i}, e_{i})-\frac{1}{6}\sum_{j}S(H_{g, h}e_{j}, e_{j})=\frac{1}{6}(S, h)_{g}-\frac{1}{3}\sum_{i}S(H_{g, h}e_{i}, e_{i})         \\
 &=&\frac{1}{6}(S, h)_{g}-\frac{1}{3}\sum_{i}S(\sum_{j}h(e_{i}, e_{j})e_{j}, e_{i})=\frac{1}{6}(S, h)_{g}-\frac{1}{3}\sum_{i, j}h(e_i, e_j)S(e_i, e_j)\\
 &=&\frac{1}{6}(S, h)_{g}-\frac{1}{3}(S, h)_{g}=-\frac{1}{6}(S, h)_{g} \,.
\end{eqnarray*}}

\end{document}